\documentclass[12pt]{amsart}
\usepackage[a4paper,margin=3cm]{geometry}

\usepackage{amsmath,amssymb,amsthm}
\usepackage{mathtools}
\usepackage{mathrsfs}

\usepackage[colorlinks=true,
            linkcolor=blue,
            citecolor=blue,
            urlcolor=blue]{hyperref}

\newtheorem{theorem}{Theorem}[section]
\newtheorem{proposition}[theorem]{Proposition}
\newtheorem{lemma}[theorem]{Lemma}
\newtheorem{corollary}[theorem]{Corollary}

\theoremstyle{definition}
\newtheorem{definition}[theorem]{Definition}

\theoremstyle{remark}
\newtheorem{remark}[theorem]{Remark}

\DeclareMathOperator{\h}{\mathbf{h}} % Metric functional
\newcommand{\abs}[1]{\lvert#1\rvert} %  Absolute value
\newcommand{\norm}[1]{\lVert#1\rVert} %  Norm 
\newcommand{\R}{\mathbb{R}}  % The real numbers.
\newcommand{\N}{\mathbb{N}}  % The natural numbers.

\title{A Weak Topology on Metric Spaces}

\author{Armando W. Guti\'{e}rrez}
\address{Department of Mathematics and Systems Analysis,
Aalto University,
Espoo, Finland}
\email{ext-armandow.gutierrez@aalto.fi}

\author{Olavi Nevanlinna}
\address{Department of Mathematics and Systems Analysis,
Aalto University,
Espoo, Finland}
\email{olavi.nevanlinna@aalto.fi}

\subjclass[2020]{Primary 54A20; Secondary 46B20, 51F30, 54A10}
\keywords{Metric functionals, metric-functional weak topology,
metric spaces, normed spaces}

\begin{document}

\begin{abstract}
We use metric functionals to construct a topology on metric spaces 
that induces $d$-weak convergence. We show that this topology is 
invariant under surjective isometries and characterize its convergence 
in terms of weak metrics. We prove that the topology is $T_1$ on every 
metric space. For normed linear spaces, we further show that it is 
Hausdorff and coarser than the classical weak topology. Finally, we 
construct a sequence in $C[0,1]$ that is unbounded in the sup-norm 
yet converges $d$-weakly, thereby correcting an assertion in an 
earlier work.
\end{abstract}

\maketitle

\section{Main Results}

In an earlier work \cite{GNweakconv}, we introduced the notion 
of $d$-weak convergence in arbitrary metric spaces $(X,d)$. 
We restricted our analysis there to sequences with the purpose of
testing and validating our ideas in a simple way. 
We verified that this notion behaves as expected in \emph{all} normed 
linear spaces; every \emph{bounded sequence} in a normed linear space 
converges $d$-weakly if and only if it converges in the classical 
weak sense.

Our goal here is to investigate the topological structure underlying 
this notion of convergence. 
We begin by recalling its formal definition.

\begin{definition}
Let $(X,d)$ be a metric space, and let $X^\diamondsuit$ be the space 
of all metric functionals on $X$. We say that a net of points 
$(x_i)_{i\in I}$ in $X$ converges $d$-weakly to a point $x \in X$ 
if for every $\h \in X^\diamondsuit$ we have 
$$
    \liminf_{i\in I}\h(x_i) \geq \h(x).
$$
\end{definition}

Our first result shows that $d$-weak convergence is induced by a 
topology.

\begin{theorem}\label{thm:thetopology}
Let $(X,d)$ be a metric space and $X^\diamondsuit$ the space of 
all metric functionals on $X$. Then, there exists a topology on $X$ 
induced by $X^\diamondsuit$, denoted by $\sigma(X,X^\diamondsuit)$,
with the following properties:
\begin{enumerate}
    \item The topology $\sigma(X,X^\diamondsuit)$ is coarser 
    than the metric topology on $X$.
    \item The topology $\sigma(X,X^\diamondsuit)$ on $X$ makes 
    every metric functional lower semi-continuous.
    \item A net in $X$ converges in the 
    topology $\sigma(X,X^\diamondsuit)$ if and only if it 
    converges $d$-weakly.
\end{enumerate}
\end{theorem}

In \cite[Conjecture 1.5]{GNweakconv}, we conjectured that $d$-weakly 
convergent sequences in normed linear spaces must be bounded. The 
conjecture was motivated by the results stated in 
\cite[Theorem 1.4]{GNweakconv}. However, the assertion concerning 
$C[0,1]$ in that theorem is incorrect. The following theorem shows 
that the conjecture is false.

\begin{theorem}\label{thm:conjfalse}
There exists a sequence $(f_n)$ in $C[0,1]$ that converges to $0$ 
in the topology $\sigma(C[0,1],C[0,1]^\diamondsuit)$ and 
$\norm{f_n}_\infty \to \infty$.
\end{theorem}

We next establish an invariance property of the topology 
$\sigma(X,X^{\diamondsuit})$ under surjective isometries.

\begin{theorem}\label{thm:invariance}
Let $(X,d)$ and $(Y,D)$ be metric spaces. Suppose that $T:X \to Y$ 
is a distance-preserving mapping, that is, 
$D(Tp,Tq) = d(p,q)$ for all $p,q \in X$. Assume that $Y=TX$. 
Then, the mapping
\[ T: (X, \sigma(X,X^\diamondsuit)) 
\to (Y, \sigma(Y,Y^\diamondsuit))\]
is a homeomorphism.
\end{theorem}

The terminology \emph{weak topology on metric spaces} is further 
justified by a connection with the notion of weak metrics, 
introduced by H. Ribeiro \cite{WeakMetric} in 1943. We begin by 
recalling the relevant definition.

\begin{definition}
Let $X$ be a nonempty set. A mapping 
$\delta: X\times X \to [0,+\infty)$ is called a \emph{weak metric} 
if for every $x,y,z \in X$ we have 
\[
 \delta(x,x)=0 
 \quad\text{and}
 \quad \delta(x,y)\leq \delta(x,z)+\delta(z,y).
\]
\end{definition}

Weak metrics generalize ordinary metrics by allowing asymmetry 
and zero distances between distinct points. That is, $\delta(x,y)$ 
and $\delta(y,x)$ may differ, and $\delta(x,y)$ may vanish even 
when $x\neq y$.

Examples of weak metrics are found in \cite{WeakMetED}. Some 
fixed point problems in weak metric spaces were investigated 
in \cite{FixedPWeakM}. 

The following theorem characterizes 
convergence in $\sigma(X,X^\diamondsuit)$ in terms of 
weak metrics.

\begin{theorem}\label{thm:weakmetrics}
Let $(X,d)$ be a metric space and $X^\diamondsuit$ the space of 
all metric functionals on $X$. Then, there exists a family of 
weak metrics $\{ \delta_{\h} \mid \h \in X^\diamondsuit \}$ 
on $X$ such that a net $(x_i)_{i\in I}$ in $X$ converges 
to $x \in X$ in $\sigma(X,X^\diamondsuit)$ if and
only if $\delta_{\h}(x,x_i)\to 0$ for every $\h \in X^\diamondsuit$.
\end{theorem}

We conclude the presentation of our main results with some further
statements concerning the topology $\sigma(X,X^\diamondsuit)$. 
The first establishes its separation properties, while the 
remaining ones compare it with the classical weak topology on 
normed linear spaces.

\begin{theorem}\label{thm:separation}
Let $X$ be a metric space. Then $\sigma(X,X^\diamondsuit)$ is 
$T_1$. Moreover, if $X$ is a normed linear space, then 
$\sigma(X,X^\diamondsuit)$ is Hausdorff.
\end{theorem}

\begin{theorem}\label{thm:comparison}
Let $X$ be a normed linear space. Then
\[
 \sigma(X,X^\diamondsuit)\subseteq\sigma(X,X^*).
\]
\end{theorem}

The inclusion shown above can be either an equality or a strict 
one, as the following results show.

\begin{corollary}\label{cor:eqtop}
Let $X$ be a normed linear space whose dual $X^*$ is strictly 
convex. Then
\[
\sigma(X,X^\diamondsuit)=\sigma(X,X^*).
\]
\end{corollary}

\begin{proposition}\label{prop:strictinclusion}
Let $X=\ell_1$. Then, we have
\[
 \sigma(X,X^\diamondsuit)
\subsetneq \sigma(X,X^*).
\]
\end{proposition}

\section{Proofs}

Let us begin by recalling the notion of a metric functional. 
Let $(X,d)$ be a metric space. Choose a base point $o \in X$ 
and define 
\[
   X^{\vee} := \{ \h_w \mid w \in X\},\quad
    \text{where}\quad
    \h_w:=d(\cdot,w) - d(o,w).
\]

Note that each element in $X^{\vee}$ is a 1-Lipschitz mapping
$X\to \R$ vanishing at the point $o$. Let $X^{\diamondsuit}$ 
denote the closure of $X^{\vee}$ in the topology of pointwise
convergence. Each element in $X^{\diamondsuit}$ is called a 
metric functional on $X$, and each element in $X^{\vee}$ is 
called an internal metric functional. For more details 
see \cite{GNweakconv}.

\subsection{Construction of the weak topology}

We now turn to the construction of the weak topology on $X$.
For a point $x \in X$, a nonempty finite subset 
$F \subseteq X^\diamondsuit$ and a real number $c>0$, define 
\[
    B(x,F,c):= \bigcap_{\h\in F}
    \{ y\in X \mid\, \h(y) > \h(x) - c\}.
\]

\begin{lemma}
The collection 
\[
\mathscr{B}:= \{ B(x,F,c) \mid x\in X,\, 
F\subseteq X^\diamondsuit,\, 
0 < \operatorname{card}(F) < \infty,\, c > 0\}
\]
is a basis for a topology on $X$.
\end{lemma}

\begin{proof}
First, note that every point $x \in X$ is contained in 
$B(x,\{\h_x\},1)$, where 
\[
 \h_x := d(\cdot,x)-d(o,x) \in X^\vee.
\]

Now, suppose that $x$ is a point in the intersection
of two elements of $\mathscr{B}$, say 
\[
 B(x_1,F_1,c_1) \qquad\text{and}\qquad B(x_2,F_2,c_2).
\] 
Then, we have
\[
 \h(x) > \h(x_1) - c_1,\quad\text{for all } \h \in F_1,
\]
and
\[
 \h(x) > \h(x_2) - c_2,\quad\text{for all } \h \in F_2.
\]
Define $F:=F_1 \cup F_2$ and
\[
    c := \min\left\{\min_{\h\in F_1}\{\h(x)-\h(x_1)+c_1\},\, 
    \min_{\h\in F_2}\{\h(x)-\h(x_2)+c_2\}\right\}. 
\]
Since $F_1$ and $F_2$ are nonempty finite subsets of 
$X^\diamondsuit$, so is $F$. By the preceding inequalities, we 
have $c>0$. Hence, $B(x,F,c)$ belongs to $\mathscr{B}$ and 
satisfies
\[
 x\in B(x,F,c)\subseteq B(x_1,F_1,c_1)\cap B(x_2,F_2,c_2).
\]
\end{proof}

\begin{definition}
We denote by $\sigma(X,X^\diamondsuit)$ the topology generated 
by the basis $\mathscr{B}$. That is, a subset $A \subseteq X$ is 
open in $\sigma(X,X^\diamondsuit)$ if and only if for every 
$x \in A$ there exists a basis element $B \in \mathscr{B}$ such 
that $x\in B \subseteq A$.    
\end{definition}

Note that a subset $A \subseteq X$ is open in 
$\sigma(X,X^\diamondsuit)$ if and only if for every $x \in A$ 
there exists a nonempty finite subset $F \subseteq X^\diamondsuit$ 
and a real number $c > 0$ such that $B(x,F,c) \subseteq A$. 

One may wonder if different base points produce different 
topologies. We show next that the choice of the base point 
is irrelevant.

\begin{proposition}
Let $p$ be a point in $X$ such that $d(o,p)> 0$. Let 
$X_p^\diamondsuit$ denote the space of all metric functionals 
on $X$ with base point $p$. Then, we have
$$
    \sigma(X,X^\diamondsuit)=\sigma(X,X_p^\diamondsuit).
$$
\end{proposition}

\begin{proof}
For every $x,y \in X$ we have 
\[ 
 d(x,y)-d(o,y)-(d(p,y)-d(o,y)) = d(x,y)-d(p,y).
\]
Then, $\h - \h(p)$ is in $X_p^\diamondsuit$ for all 
$\h \in X^\diamondsuit$. Now, if the points $o$ and $p$ are 
permuted in the previous equality, we conclude that the mapping 
$\h\mapsto \h - \h(p)$ is a bijection from $X^\diamondsuit$ 
onto $X_p^\diamondsuit$. 

Since subtracting a constant does not affect the inequalities 
defining the elements of the basis $\mathscr{B}$, the two 
topologies coincide.
\end{proof}

\begin{lemma}\label{lem:coarse}
The topology $\sigma(X,X^\diamondsuit)$ is coarser than the 
metric topology on $(X,d)$. 
\end{lemma}

\begin{proof}
Let $A$ be an open set in the topology $\sigma(X,X^\diamondsuit)$
and $x \in A$. Then, there exists a basis element 
$B(x,F,c)\subseteq A$. Since every metric functional is 
$1$-Lipschitz, we have
\[
 U(x,c):=\{y\in X \mid d(x,y) < c\} \subseteq B(x,F,c).
\]
Since $U(x,c)$ is an open metric ball, it follows that $A$ is 
open in the metric topology.
\end{proof}

Metric functionals are continuous as mappings from $X$ to $\R$ 
with respect to their metric topologies. When $X$ is equipped 
with the topology $\sigma(X,X^\diamondsuit)$, we have the 
following weaker property.

\begin{lemma}\label{lem:lowersc}
The topology $\sigma(X,X^\diamondsuit)$ makes every metric 
functional lower semi-continuous on $X$.  
\end{lemma}

\begin{proof}
Let $\h$ be a metric functional on $X$. For every real number
$r$, the set
\[
 A_r \coloneqq \{y\in X \mid \h(y) > r\} 
\]
is open in the topology $\sigma(X,X^\diamondsuit)$. Indeed,
for every $x\in A_r$, we have
\[
 B(x,\{\h\}, \h(x)-r) \subseteq A_r.
\]
\end{proof}

We are now ready to complete the proof of 
Theorem~\ref{thm:thetopology}.

\begin{proof}[\bf Proof of Theorem~\ref{thm:thetopology}]
The first property is shown in Lemma~\ref{lem:coarse} and the 
second property in Lemma~\ref{lem:lowersc}. It remains to 
prove the third assertion.

Suppose that a net $(a_i)_{i\in I}$, indexed by a directed 
set $(I,\succcurlyeq)$, converges $d$-weakly to a point 
$a \in X$. Let $A$ be an open set in the topology 
$\sigma(X,X^\diamondsuit)$ with $a \in A$. Then, there exists 
a real number $c > 0$ and a nonempty finite subset 
$F \subseteq X^\diamondsuit$ such that 
\[
 B(a,F,c) \subseteq A.
\] 
By $d$-weak convergence, for each metric functional $\h \in F$ 
there exists an index $i_0(c,\h)\in I$  such that 
\[
    \h(a_i) > \h(a) - c,
\]
for all $i \succcurlyeq i_0(c,\h)$.
Since $F$ is a nonempty finite set, there exists an index 
$i_*\in I$ such that $i_* \succcurlyeq i_0(c,\h)$ for all 
$\h \in F$. Thus, for every $i \succcurlyeq i_*$, we have 
\[
  a_i \in B(a,F,c) \subseteq A.  
\]
This shows that $(a_i)_{i\in I}$ converges to $a$ in the 
topology $\sigma(X,X^\diamondsuit)$.

Now, suppose that a net $(a_i)_{i\in I}$ converges to a point
$a \in X$ in the topology $\sigma(X,X^\diamondsuit)$. 
Let $\h$ be a metric functional on $X$ and $c$ a positive real 
number. Since $B(a,\{\h\},c)$ is open in 
$\sigma(X,X^\diamondsuit)$ and contains the point $a$, there 
exists an index $i_0 \in I$ such that 
\[
    a_i \in B(a,\{\h\},c),
\]
for all $i \succcurlyeq i_0$. Thus, we have
\[
    \liminf_{i\in I} \h(a_i) \geq \h(a).
\]
\end{proof}

\subsection{Escaping sequence that converges weakly}

We next establish Theorem~\ref{thm:conjfalse} by constructing 
an unbounded $d$-weakly convergent sequence in $C[0,1]$.

\begin{proof}[\bf Proof of Theorem~\ref{thm:conjfalse}]
Choose a sequence of nondegenerate closed intervals
\[
 A_n\subseteq[0,1]\quad\text{with}\quad 
 A_n\cap A_m = \varnothing
 \quad\text{for all } n\neq m. 
\]
For each $n\geq 1$, choose a nonnegative 
continuous function $f_n$ on $[0,1]$ that vanishes outside 
$A_n$ and attains its maximum value $n$. Then, for all $n\geq 1$, 
\[
 \norm{f_n}_\infty = n.
\]

We show next that for every $\h \in C[0,1]^\diamondsuit$ we 
have
\[
 \liminf_{n}\h(f_n) \geq 0.
\]

First, we claim that for each internal $\h_w \in C[0,1]^\vee$,
where 
\[
 \h_w \coloneqq \norm{\cdot - w}_\infty-\norm{w}_\infty,
\] 
the value $\h_w(f_n)$ is negative for at most one $n$. Indeed, 
for a fixed $w\in C[0,1]$ there exists $s\in [0,1]$ such that
$\abs{w(s)}=\norm{w}_\infty$. If $s \notin A_n$, then $f_n(s)=0$ 
and 
\[
    \h_w(f_n)=\norm{f_n - w}_\infty - \norm{w}_\infty 
    \geq \abs{w(s)}-\abs{w(s)}=0.
\]
Since $A_n$ and $A_m$ are disjoint for all $n\neq m$, the
value $\h_w(f_n)$ is negative for at most one $n$.

Now, let $\h$ be an arbitrary metric functional on $C[0,1]$.
Then, there exists a net $(w_i)_{i\in I}$ in $C[0,1]$ such that
$\h_{w_i}(f) \to \h(f)$ for all $f \in C[0,1]$. Suppose that 
there exist two distinct positive integers $n$ and $m$ such that
\[
 \h(f_n) < 0 \quad \text{and} \quad \h(f_m) < 0.
\] 
Choose $\epsilon >0$ such that
\[
 \h(f_n)< -2\epsilon \quad \text{and}\quad \h(f_m)< -2\epsilon.
\] 
Since $\h_{w_i} \to \h$ pointwise, there exist $j(\epsilon,n)$ 
and $k(\epsilon,m)$ in the directed set $(I,\succcurlyeq)$ such 
that 
\[
 \h_{w_i}(f_n) < \h(f_n) + \epsilon
 \quad\text{for all } i\succcurlyeq j(\epsilon,n)
\] 
and 
\[
 \h_{w_i}(f_m) < \h(f_m) + \epsilon 
 \quad\text{for all } i\succcurlyeq k(\epsilon,m).
\]
Choose $i_0 \in I$ such that $i_0 \succcurlyeq j(\epsilon,n)$ 
and $i_0 \succcurlyeq k(\epsilon,m)$. Then, for every 
$i\succcurlyeq i_0$, 
\[
 \max\{\h_{w_i}(f_n),\h_{w_i}(f_m)\} < -\epsilon\,.
\]
This contradicts the preceding argument, which shows that each 
internal metric functional $\h_{w_i}$ can take a negative value 
for at most one $f_n$. Therefore, $\h(f_n) < 0$ for at most one 
$n$. Consequently, 
\[
 \liminf_{n}\h(f_n) \geq 0 =\h(0).
\]
\end{proof}

\begin{remark}
We listed in \cite[Theorem 1.4]{GNweakconv} some normed linear 
spaces in which $d$-weak convergence of sequences implies 
boundedness. However, our argument for the case $C[0,1]$ contains 
an error, as demonstrated by the sequence $(f_n)$ constructed 
above.

In $\ell_1$ and in normed linear spaces with strictly convex 
duals, $d$-weakly convergent sequences are bounded.
\end{remark}

\subsection{Invariance under surjective isometries}

We now turn to the invariance of the topology 
$\sigma(X,X^\diamondsuit)$ under surjective isometries.

\begin{proof}[\bf Proof of Theorem~\ref{thm:invariance}]
Choose $o \in X$ and $To \in Y$ as the base points in the 
constructions of $X^\diamondsuit$ and $Y^\diamondsuit$, 
respectively. Since $T$ is surjective and distance-preserving, 
for every point $y \in Y$ there exists a point $w \in X$ such 
that $y=Tw$ and 
\[
    D(Tx,y)-D(To,y) = d(x,w)-d(o,w)
    \quad\text{for all } x \in X.
\]
Then, the mapping 
\[ 
 T^\diamondsuit : Y^\diamondsuit \to X^\diamondsuit 
 \quad \h \mapsto \h\circ T
\]
is a bijection whose inverse is given by 
\[
 (T^\diamondsuit)^{-1}(\h)=\h\circ T^{-1}
 \quad\text{for all } \h \in X^\diamondsuit.
\]

If a net $(x_i)_{i\in I}$ in $X$ converges $d$-weakly
to $x \in X$, then for every $\h \in Y^\diamondsuit$,
\begin{align*}
    \liminf_{i\in I}\,\h(Tx_i) 
    &= \liminf_{i\in I}\,T^\diamondsuit\h(x_i)\\
    &\geq T^\diamondsuit\h(x)\\
    &= \h(Tx).   
\end{align*}
Thus $(Tx_i)_{i\in I}$ converges $D$-weakly to $Tx \in Y$.

Conversely, if a net $(y_i)_{i\in I}$ in $Y$ converges 
$D$-weakly to $y \in Y$, then for every $\h \in X^\diamondsuit$,
\begin{align*}
    \liminf_{i\in I}\,\h(T^{-1}y_i) 
    &= \liminf_{i\in I}\,((T^\diamondsuit)^{-1}\h)(y_i)\\
    &\geq ((T^\diamondsuit)^{-1}\h)(y)\\
    &= \h(T^{-1}y).    
\end{align*}
Hence $(T^{-1}y_i)_{i\in I}$ converges $d$-weakly to 
$T^{-1}y \in X$.

By Theorem~\ref{thm:thetopology}, the preceding arguments show 
that $T$ is a homeomorphism between $(X,\sigma(X,X^\diamondsuit))$ 
and $(Y,\sigma(Y,Y^\diamondsuit))$.
\end{proof}

\subsection{Weak metrics}
For each $\h \in X^\diamondsuit$ define the mapping 
$\delta_{\h} : X\times X \to [0,+\infty)$ by
\[
    (x,y) \mapsto \delta_{\h}(x,y) 
    \coloneqq \max\{0, \h(x) - \h(y)\}.
\]

\begin{lemma}
For every $\h \in X^\diamondsuit$, $\delta_{\h}$ is a weak 
metric on $X$.    
\end{lemma}

\begin{proof}
Recall that for every pair of real numbers $a$ and $b$ we have
\[
 0 \leq \max\{0,a+b\} \leq \max\{0,a\} + \max\{0,b\}.
\]
\end{proof}

\begin{proof}[\bf Proof of Theorem~\ref{thm:weakmetrics}]
By Theorem~\ref{thm:thetopology}, a net $(x_i)_{i\in I}$ in $X$ 
converges to $x \in X$ in the topology $\sigma(X,X^\diamondsuit)$
if and only if for every $\h \in X^\diamondsuit$ and
every real number $\epsilon > 0$ there exists an index $i_0 \in I$ 
such that
\[
    \h(x_i) > \h(x) - \epsilon \qquad\text{for all }
    i\succcurlyeq i_0.
\]
This is equivalent to the property $\delta_{\h}(x,x_i)\to 0$ for 
all $\h \in X^\diamondsuit$.
\end{proof}

\subsection{Separation properties}

\begin{proof}[\bf Proof of Theorem~\ref{thm:separation}]

Let $(X,d)$ be a metric space and let $x,y\in X$ be distinct 
points. Recall that
\[
 \h_x,\h_y\in X^\vee\subseteq X^\diamondsuit.
\]
Since $\h_y(x)-\h_y(y)=d(x,y) > 0$, we have
\[
 y\notin B(x,\{\h_y\},d(x,y)).
\]
Similarly, $\h_x(y)-\h_x(x)=d(x,y) > 0$ implies
\[
 x\notin B(y,\{\h_x\},d(x,y)).
\]
Thus the topology $\sigma(X,X^\diamondsuit)$ is $T_1$.

Suppose now that $X$ is a normed linear space. By the Hahn--Banach 
and Krein--Milman theorems, the extreme points of the closed unit 
ball of $X^*$ separate points of $X$. Thus, for distinct $x,y\in X$, 
there exists an extreme point $f$ such that $f(x)<f(y)$. Since the 
dual unit ball is symmetric, $-f$ is also an extreme point.
By \cite[Corollary 3.5]{Wal18}, both $f$ and $-f$ belong to 
$X^\diamondsuit$.

Define
\[
 c\coloneqq \frac{f(y)-f(x)}{2}>0.
\]
Then
\[
 B(x,\{-f\},c)=\{z\in X \mid f(z)<f(x)+c\}
\]
and
\[
 B(y,\{f\},c)=\{z\in X \mid f(z)>f(y)-c\}
\]
are open neighborhoods of $x$ and $y$, respectively. Since
\[
 f(x)+c=f(y)-c,
\]
these neighborhoods are disjoint. Hence $\sigma(X,X^\diamondsuit)$ 
is Hausdorff.
\end{proof}

The topology $\sigma(X,X^\diamondsuit)$ need not be Hausdorff for 
general metric spaces. For example, let $X=\R$ be equipped with 
the metric
\[
 d(x,y)\coloneqq \sqrt{\abs{x-y}}.
\]
Then $X^\diamondsuit=X^\vee\cup\{\mathbf{0}\}$, and the sequence 
$(a_n)$ defined by $a_n\coloneqq n$ converges to every point of 
$X$ in the topology $\sigma(X,X^\diamondsuit)$.

\subsection{Comparison with the classical weak topology}

\begin{proof}[\bf Proof of Theorem~\ref{thm:comparison}]

Every metric functional is convex, because internal metric 
functionals are convex and convexity is preserved under pointwise 
limits. Moreover, every metric functional is $1$-Lipschitz and 
therefore norm-continuous. Consequently, for each 
$\h\in X^\diamondsuit$ and $a\in\R$, the set
\[
 \{x\in X \mid \h(x)\le a\}
\]
is convex and norm-closed, hence closed in the classical weak 
topology $\sigma(X,X^*)$ by the Hahn--Banach separation theorem. 
Thus, the set $\{x\in X \mid \h(x)>a\}$ is open in $\sigma(X,X^*)$. 
Since strict superlevel sets of metric functionals generate the 
topology $\sigma(X,X^\diamondsuit)$, the asserted inclusion follows.
\end{proof}

\begin{proof}[\bf Proof of Corollary~\ref{cor:eqtop}]

By Theorem~\ref{thm:comparison}, it suffices to prove the inclusion
\[
 \sigma(X,X^*)\subseteq\sigma(X,X^\diamondsuit).
\]

Since $X^*$ is strictly convex, every element of its unit sphere 
is an extreme point of its closed unit ball. 
By \cite[Corollary 3.5]{Wal18}, every such functional belongs 
to $X^\diamondsuit$. Consequently, every element of $X^*$ is
continuous with respect to $\sigma(X,X^\diamondsuit)$.
Thus
\[
 \sigma(X,X^*)\subseteq\sigma(X,X^\diamondsuit).
\]
\end{proof}

\begin{proof}[\bf Proof of Proposition~\ref{prop:strictinclusion}]

By \cite[Theorem 3.6]{Gut19}, every $\h\in X^\diamondsuit$ has the 
form
\[
 \h(x)=\sum_{n\in I}\varepsilon_n x_n
 +\sum_{n\notin I}\bigl(\abs{x_n-z_n}-\abs{z_n}\bigr),
\]
where $I\subseteq\N$, $\varepsilon_n\in\{-1,1\}$ 
for $n\in I$, and $z_n\in\R$ for $n\notin I$.

For each $\h\in X^\diamondsuit$, define $s=(s_n)\in\{-1,1\}^{\N}$ 
by
\[
 s_n \coloneqq
\begin{cases}
\varepsilon_n, & n\in I,\\
-\operatorname{sgn}(z_n), & n\notin I,\ z_n\ne 0,\\
1, & n\notin I,\ z_n=0.
\end{cases}
\]
Since, for all $a,b\in\R$ with $b\ne 0$,
\[
 \abs{a-b} \ge \operatorname{sgn}(b)(b-a)
 =\abs{b} - \operatorname{sgn}(b)a,
\]
and $\abs{a} \ge a$, it follows that, for every $x\in X$,
\[
 \h(x) \ge \sum_{n=1}^{\infty}s_nx_n.
\]

We now show that the functional $f\in X^*$ given by
\[
 f(x)=\sum_{n=1}^\infty\frac{x_n}{n}
\]
is not continuous in the topology $\sigma(X,X^\diamondsuit)$.
Suppose otherwise. Define
\[
 U \coloneqq \{x\in X \mid f(x)<1\}.
\]
Then $U$ is an open neighborhood of $0$ in the topology 
$\sigma(X,X^\diamondsuit)$. Hence there exist
\[ 
 \h^1,\ldots,\h^m\in X^\diamondsuit\quad\text{and}\quad c>0
\]
such that
\[
 B(0,\{\h^1,\ldots,\h^m\},c)\subseteq U.
\]
For each $1\le j\le m$, choose a sign sequence $s^j=(s_n^j)$ 
satisfying
\[
 \h^j(x) \ge \sum_{n=1}^\infty s_n^j x_n.
\]
for all $x\in X$. Since the sequence
\[
 v_n \coloneqq (s_n^1,\ldots,s_n^m),\qquad n\in\N,
\]
takes values in the finite set $\{-1,1\}^m$, there exist distinct 
$p,q\in\N$ such that $v_p=v_q$, that is,
\[
 s_p^j=s_q^j
\qquad\text{for }j=1,\ldots,m.
\]

Let $(e_n)$ be the canonical basis of $X$ and define
\[
 x \coloneqq \frac{e_p-e_q}{1/p-1/q}.
\]
Then $f(x)=1$, so $x\notin U$. On the other hand, for every 
$j=1,\ldots,m$, we have
\[
 \h^j(x)\ge
 \frac{s_p^j-s_q^j}{1/p-1/q}
 =0>-c.
\]
Thus
\[
 x\in B(0,\{\h^1,\ldots,\h^m\},c)
 \subseteq U,
\]
a contradiction. 

Therefore $f$ is not continuous in the topology 
$\sigma(X,X^\diamondsuit)$, hence we have
\[
 \sigma(X,X^\diamondsuit)
 \subsetneq \sigma(X,X^*).
\]
\end{proof}

% Bibliography
\bibliographystyle{amsplain}
\bibliography{ref}

@article{GNweakconv,
    author = {Guti{\'e}rrez, Armando W. and Nevanlinna, Olavi},
    title = {Metric functionals and weak convergence},
    journal = {Z. Anal. Anwend.},
    year = {2026},
    doi = {10.4171/ZAA/1828}
}

@article{WeakMetED,
 author = {Papadopoulos, Athanase and Troyanov, Marc},
 title = {Weak metrics on {Euclidean} domains},
 fjournal = {JP Journal of Geometry and Topology},
 journal = {JP J. Geom. Topol.},
 issn = {0972-415X},
 volume = {7},
 number = {1},
 pages = {23--43},
 year = {2007},
 zbMATH = {5233854},
 Zbl = {1189.30083}
}

@article{WeakMetric,
 author = {Ribeiro, Hugo},
 title = {Sur les espaces {\`a} m{\'e}trique faible},
 fjournal = {Portugaliae Mathematica},
 journal = {Port. Math.},
 issn = {0032-5155},
 volume = {4},
 pages = {21--40},
 year = {1943},
 url = {https://eudml.org/doc/114615},
 zbMATH = {3044387},
 Zbl = {0028.19103}
}

@article{FixedPWeakM,
 author = {Guti{\'e}rrez, Armando W. and Walsh, Cormac},
 title = {Firm non-expansive mappings in weak metric spaces},
 fjournal = {Archiv der Mathematik},
 journal = {Arch. Math.},
 issn = {0003-889X},
 volume = {119},
 number = {4},
 pages = {389--400},
 year = {2022},
 doi = {10.1007/s00013-022-01759-5},
 zbMATH = {7589462}
}

@article{Wal18,
  author  = {Walsh, Cormac},
  title   = {Hilbert and {Thompson} geometries isometric
             to infinite-dimensional {Banach} spaces},
  journal = {Annales de l'Institut Fourier},
  volume  = {68},
  number  = {5},
  pages   = {1831--1877},
  year    = {2018},
  doi     = {10.5802/aif.3198}
}

@article{Gut19,
  author  = {Guti{\'e}rrez, Armando W.},
  title   = {On the Metric Compactification of
             Infinite-dimensional {$\ell_p$} Spaces},
  journal = {Canadian Mathematical Bulletin},
  volume  = {62},
  number  = {3},
  pages   = {491--507},
  year    = {2019},
  doi     = {10.4153/S0008439518000681}
}

\end{document}